\documentclass{amsart}
\usepackage{amscd,amsmath,amssymb,amsthm,bm,cases,color,comment,mathrsfs,mathtools}
\usepackage[dvips]{graphicx}
\usepackage[all]{xy}

\theoremstyle{definition}
\newtheorem{definition}{Definition}[section]

\theoremstyle{plain}
\newtheorem{theorem}[definition]{Theorem}
\newtheorem{proposition}[definition]{Proposition}

\numberwithin{equation}{section}

\title[Distortion of the Torelli groups with boundary components]{A lower bound of the distortion of the Torelli group in the mapping class group with boundary components}

\author[E.~Kuno]{Erika Kuno}
\address{
(Erika Kuno)
Department of Mathematics,
Tokyo Institute of Technology,
2-12-1 Oh-okayama, Meguro-ku, Tokyo 152-8551, Japan
}
\email{kuno.e.aa@m.titech.ac.jp}
\author[G.~Omori]{Genki Omori}
\address{
(Genki Omori)
Department of Mathematics,
Tokyo Institute of Technology,
Oh-okayama, Meguro, Tokyo 152-8551, Japan
}
\email{omori.g.aa@m.titech.ac.jp}
\date{\today}

\begin{document}

\begin{abstract}
We prove that each Torelli group of an orientable surface with any number of boundary components is at least exponentially distorted in the mapping class group by using Broaddus-Farb-Putman's techniques.
Further we show that the distortion of each Torelli group in the level $d$ mapping class group is the same as that of in the mapping class group. 
\end{abstract}

\maketitle

\section{Introduction}\label{Introduction}

Let $S_{g, b}$ be a compact connected orientable surface of genus $g\geq 3$ with $b\geq 0$ boundary components.
The {\it mapping class group} $\mathcal{M}_{g, b}$ of $S_{g, b}$ is the group consists of the isotopy classes of orientation preserving homeomorphisms on $S_{g, b}$, fixing the boundary components pointwise.
Let $i\colon S_{g, b}\rightarrow S_{g, 0}$ be a natural inclusion map defined by capping the boundary components of $S_{g, b}$ by $b$ disks.
Then there is a surjective homomorphism $\varphi\colon \mathcal{M}_{g, b}\rightarrow\mathcal{M}_{g, 0}$ induced by $i$, namely, $\varphi(f)=[\bar{F}]$ ($f=[F]\in\mathcal{M}_{g, b}$), where $\bar{F}(x)=F(x)$ if $x\in S_{g, b}$ and $\bar{F}(x)=x$ if $x\not\in S_{g, b}$.
For $b=0,1$, the {\it Torelli group} ${\mathcal I}_{g, b}$ of $S_{g, b}$ is the kernel of the natural homomorphism $\Phi \colon {\mathcal M}_{g, b}\rightarrow {\rm Aut}(H_{1}(S_{g, b};{\mathbb Z}))$.
We note that for $b=0,1$ the {\it Torelli group} is a normal subgroup of the mapping class group.
For $b\geq 2$, we define the {\it Torelli group} ${\mathcal I}_{g, b}$ of $S_{g, b}$ by $\varphi^{-1}(\mathcal{I}_{g, 0})$.
By the definition, $\mathcal{I}_{g,b}$ is also a normal subgroup of $\mathcal{M}_{g,b}$ for $g\geq 2$. 
There are several definitions of Torelli groups with $b\geq 2$ boundary components (see~\cite{Putman11}).
In case of our definition, the Torelli group is finitely generated (it is a consequence of \cite[Theorem 4.1]{Putman11}), and hence it is equipped with a word norm.
The geometry of the mapping class groups of orientable surfaces is well understood.
Therefore, understanding the geometry of the inclusion homomorphism ${\mathcal I}_{g, b}\hookrightarrow\mathcal{M}_{g, b}$ may allow one to deduce geometric
properties of the Torelli group from geometric properties of the mapping class group.

Let $G$ be a finitely generated group and $K$ a finitely generated subgroup of $G$.
In this paper, we denote by $\parallel\cdot\parallel_{G}$ a word metric of a finitely generated group $G$. 
Then, there exists $C>0$ such that $\parallel h\parallel_{G}\leq C\parallel h\parallel_{K}$ for any $h\in K$.
A natural question is what the smallest function $\delta\colon{\mathbb N}\rightarrow{\mathbb R}$ which satisfies $\parallel h\parallel_{K}\leq\delta(\parallel h\parallel_{G})$ is.
The distortion of $K$ in $G$ is at most $\delta$ if there exists $C$, $C'$ such that for each $h\in K$, it follows that $\parallel h\parallel_{K}\leq C\delta(\parallel h\parallel_{G})+C'$.
The distortion of $K$ in $G$ is at least $\delta$ if there exists a sequence $\{h_{i}\}$ ($h_{i}\in K$) such that the word norm of $h_{i}$ in $G$ grows linearly, while the word norm in $K$ grows $\delta$.
If $K$ is at most and at least $\delta$ distorted in $G$, then we say the {\it distortion} of $K$ in $G$ is (exactly) $\delta$.
If the distortion of $K$ in $G$ is linear, we call $K$ is {\it undistorted} in $G$.

The distortions of various subgroups in the mapping class groups have been investigated.
For example, the mapping class groups of subsurfaces are undistorted in the mapping class group by Masur-Minsky~\cite{Masur-Minsky00} and Hamenst\"{a}dt~\cite{Hamenstad09a}, or the handle body group is exponentially distorted in the corresponding mapping class group by Hamenst\"{a}dt-Hensel~\cite{Hamenstad-Hensel12}.
Moreover, Broaddus-Farb-Putman~\cite{Broaddus-Farb-Putman11} proved that $\mathcal{I}_{g, 0}$ (resp. $\mathcal{I}_{g, 1}$) is at least exponentially and at most double exponentially distorted in $\mathcal{M}_{g, 0}$ (resp. $\mathcal{M}_{g, 1}$).
From the result of Cohen~\cite{Cohen14}, it follows that the distortion of $\mathcal{I}_{g, 0}$ (resp. $\mathcal{I}_{g, 1}$) in $\mathcal{M}_{g, 0}$ (resp. $\mathcal{M}_{g, 1}$) is exponential.
We find a lower bound of the distortion of the Torelli groups in the mapping class groups for the orientable surfaces with $b\geq2$ boundary components by applying Broaddus-Farb-Putman's arguments:

\begin{theorem}\label{first_thm}
For $g\geq 3$ and $b\geq 2$, ${\mathcal I}_{g, b}$ is at least exponentially distorted in ${\mathcal M}_{g, b}$.
\end{theorem}

For $d\geq 2$, the {\it level $d$ mapping class group} ${\mathcal M}_{g, b}[d]$ of $S_{g, b}$ ($b=0, 1$) is the kernel of the natural homomorphism $\Phi_{d}\colon {\mathcal M}_{g, b}\rightarrow {\rm Aut}(H_{1}(S_{g, b};{\mathbb Z}/d{\mathbb Z}))$.
We show that the distortion of the Torelli groups in the level $d$ mapping class group is the same as that of in the mapping class groups, and so we obtain the following.

\begin{theorem}\label{second_thm}
For $g\geq 3$, $b\leq 1$ and $d\geq 2$, ${\mathcal I}_{g, b}$ is exponentially distorted in ${\mathcal M}_{g, b}[d]$.
\end{theorem}

We note that we cannot apply Broaddus-Farb-Putman's arguments to the upper bound of the distortion of ${\mathcal I}_{g, b}$ in ${\mathcal M}_{g, b}$ for $g\geq 3$ and $b\geq 2$.
The reason is as follows.
For $b\leq 1$, the image $\Phi ({\mathcal M}_{g, b})$ is isomorphic to the symplectic group $\mathrm{Sp}(2g;\mathbb{Z})$, and they use a property of $\mathrm{Sp}(2g;\mathbb{Z})$ for the upper bound of the distortion of ${\mathcal I}_{g, b}$ in ${\mathcal M}_{g, b}$.
However, for $b\geq 2$ we don't know whether the image $\Phi ({\mathcal M}_{g, b})$ is isomorphic to some symplectic group, and so we cannot use the same approach as that of Broaddus-Farb-Putman.

\section{Proof of Theorem~\ref{first_thm}}\label{Proof of First theorem}

In this section, we prove Theorem~\ref{first_thm}. Firstly, we give the definition of the partially hyperbolic matrices.

\begin{definition}
Let $V$ be a free abelian group.
We say that an element of the automorphism group $\mathrm{GL}(V)\cong \mathrm{Aut}(V)$ is {\it partially hyperbolic} if the corresponding linear transformation of $V\otimes\mathbb{C}$ has some eigenvalue $\lambda$ with $|\lambda|>1$.
We call such a matrix a {\it partially hyperbolic matrix}.
\end{definition}

We use the following proposition by Broaddus-Farb-Putman~\cite{Broaddus-Farb-Putman11} to prove Theorem~\ref{first_thm}.

\begin{proposition}{\rm (}\cite[Proposition 2.3]{Broaddus-Farb-Putman11}{\rm )}\label{criterion for exponential distortion}
Let $G$ be a finitely generated group and $K$ be a finitely generated subgroup of $G$.
Suppose that $V$ is a free abelian subgroup equipped with a $G$-action $\rho\colon G\rightarrow {\rm GL}(V)$ and that $\psi\colon K\rightarrow V$ is a surjective homomorphism which is $G$-equivariant, where $G$ acts on $K$ by conjugation.
If $\rho(G)$ contains a partially hyperbolic matrix, then the distortion of $K$ in $G$ is at least exponential.
\end{proposition}

We assume that $g\geq 3$.
We set $H=H_{1}(S_{g, 0};{\mathbb Z})$.
Broaddus-Farb-Putman~\cite{Broaddus-Farb-Putman11} proved that $G=\mathcal{M}_{g, 0}$, $K=\mathcal{I}_{g, 0}$, $V=\wedge^{3}H/H$, and $\psi\colon K\rightarrow V$ satisfy the assumptions in Proposition~\ref{criterion for exponential distortion}, where $\psi$ is the Johnson homomorphism.
Hence $\mathcal{I}_{g, 0}$ is at least exponentially distorted in $\mathcal{M}_{g, 0}$.
We use this fact in the proof.

\begin{proof}[Proof of Theorem~\ref{first_thm}]
We set $G=\mathcal{M}_{g, b}$ and $K=\mathcal{I}_{g, b}$ ($g\geq 3$, $b\geq 2$).
By the definition of $\mathcal{I}_{g, b}$, $\mathcal{M}_{g, b}$ acts on $\mathcal{I}_{g, b}$ by conjugation.
Further, we put $V=\wedge^{3}H/H$.
We regard $S_{g,0}$ as the surface obtained by capping $S_{g,b}$ by $b$ disks.
Let $\varphi\colon\mathcal{M}_{g, b}\rightarrow\mathcal{M}_{g, 0}$ be the homomorphism defined in Section~\ref{Introduction}.
Let $\tau\colon\mathcal{I}_{g, 0}\rightarrow \wedge^{3}H/H$ be the Johnson homomorphism, and $\psi$ the composition of $\varphi|_{K}$ with $\tau$.
The homomorphism $\psi$ is surjective since $\varphi$ induces the surjective homomorphism $\varphi|_{K}\colon K\rightarrow \mathcal{I}_{g, 0}$.
We denote by $\rho$ a homomorphism defined by the composition of $\varphi$ with the homomorphism $\varphi'\colon\mathcal{M}_{g, 0}\rightarrow\mathrm{GL}(\wedge^{3}H/H)$.
Since $\varphi$ is surjective and Broaddus-Farb-Putman~\cite{Broaddus-Farb-Putman11} showed that ${\rm Im}(\varphi')$ contains a partially hyperbolic matrix, ${\rm Im}(\rho)$ also contains a partially hyperbolic matrix.

We show that they satisfy the assumptions of Proposition~\ref{criterion for exponential distortion}.
It is sufficient to prove the surjective homomorphism $\psi\colon\mathcal{I}_{g, b}\rightarrow\wedge^{3}H/H$ is $\mathcal{M}_{g, b}$-equivariant, namely, $f\cdot\psi(g)=\psi(f\cdot g)$ for any $f\in\mathcal{M}_{g,b}$ and $g\in\mathcal{I}_{g, b}$.
Since $f\cdot\psi(g)=(\varphi(f))(\psi(g))$ and $\psi(f\cdot g)=\psi(fgf^{-1})$, we show $(\varphi(f))(\psi(g))=\psi(fgf^{-1})$.
Since $\tau$ is $\mathcal{M}_{g, b}$-equivariant, we have $(\varphi(f))(\psi(g))=\tau(\varphi(f)\cdot\varphi(g))=\tau(\varphi(fgf^{-1}))=\psi(fgf^{-1})$.
Hence, $\psi$ is $\mathcal{M}_{g, b}$-equivariant, and we have finished the proof.
\end{proof}

\section{Proof of Theorem~\ref{second_thm}}\label{Proof of Second theorem}

In this section we prove Theorem~\ref{second_thm}.

\begin{proof}[Proof of Theorem~\ref{second_thm}]
We can show ${\mathcal M}_{g, b}[d]$ has a finite index in ${\mathcal M}_{g, b}$.
Then, ${\mathcal M}_{g, b}[d]$ is quasi-isometrically embedded in ${\mathcal M}_{g, b}$, that is, there exists $\lambda>0$ such that $\parallel f\parallel_{{\mathcal M}_{g, b}}\leq\lambda\parallel f\parallel_{{\mathcal M}_{g, b}[d]}+\lambda$ for any $f\in{\mathcal M}_{g, b}[d]$.
From the fact that the distortion of ${\mathcal I}_{g, b}$ in ${\mathcal M}_{g, b}$ is exponential, there exist $C_{1}, C_{2}>0$ such that
$\parallel f\parallel_{{\mathcal I}_{g, b}}\leq C_{1}e^{\parallel f\parallel_{{\mathcal M}_{g, b}}}+C_{2}$ for any $f\in{\mathcal I}_{g, b}$.
Hence, it follows that $\parallel f\parallel_{{\mathcal I}_{g, b}}\leq C_{1}e^{\lambda\parallel f\parallel_{{\mathcal M}_{g, b}[d]}+\lambda}+C_{2}$.
Then we see the distortion of ${\mathcal I}_{g, b}$ in ${\mathcal M}_{g, b}[d]$ is at most exponential.
If $L\leq K\leq G$, then the distortion of $L$ in $G$ is at most the composition of the distortion of $L$ in $K$ with the distortion of $K$ in $G$.
We know ${\mathcal I}_{g, b}\leq {\mathcal M}_{g, b}[d]\leq {\mathcal M}_{g, b}$.
Now ${\mathcal M}_{g, b}[d]$ has finite index in ${\mathcal M}_{g, b}$.
Thus ${\mathcal M}_{g, b}[d]$ is undistorted in ${\mathcal M}_{g, b}$.
Therefore the distortion of ${\mathcal I}_{g, b}$ in ${\mathcal M}_{g, b}[d]$ is at least exponential, and we are done.
\end{proof}

\par
{\bf Acknowledgements:} The authors are deeply grateful to Hisaaki Endo for his warm encouragement and helpful advice.
The first author was supported by JSPS KAKENHI, the grant number 16J00397 and the second author was supported by JSPS KAKENHI, the grant number 15J10066 of Research Fellowship for Young Scientists.


\end{document}